\documentclass[11pt]{articlefederico}

\usepackage{todonotes,xypic}
\usepackage{amsthm}
\usepackage{amsmath}
\usepackage{amssymb}
\usepackage{color}\bigskip
\usepackage{enumerate}
\usepackage{bbm}
\usepackage{a4wide}

\usepackage{tocloft}

\newtheorem{theorem}{Theorem}[section]

\newtheorem{proposition}[theorem]{Proposition}

\newtheorem{definition}[theorem]{Definition}

\theoremstyle{definition}
\newtheorem{example}[theorem]{Example}
\newtheorem{remark}[theorem]{Remark}

\makeatletter
\newtheorem*{rep@theorem}{\rep@title}
\newcommand{\newreptheorem}[2]{%
\newenvironment{rep#1}[1]{%
 \def\rep@title{#2 \ref{##1}}%
 \begin{rep@theorem}}%
 {\end{rep@theorem}}}
 \makeatother
 
\newreptheorem{theorem}{Theorem}
\newreptheorem{corollary}{Corollary}
\newreptheorem{conjecture}{Conjecture}

\usepackage{todonotes}

\begin{document}
\title{\textsf{Tutte polynomials of hyperplane arrangements \\ and the finite field method.}}
\author{
\textsf{Federico Ardila\footnote{\noindent \textsf{San Francisco State University; Mathematical Sciences Research Institute; U. de Los Andes; federico@sfsu.edu.}
\newline 
The author was supported by NSF CAREER Award DMS-0956178, NSF Combinatorics Award DMS-1600609, and NSF Award DMS-1440140 to the  Mathematical Sciences Research Institute in Berkeley, California.}}}
\date{}
\maketitle

\begin{abstract} 
The Tutte polynomial is a fundamental invariant associated to a graph, matroid, vector arrangement, or hyperplane arrangement.
This short survey focuses on some of the most important results on Tutte polynomials of hyperplane arrangements.
We show that many enumerative, algebraic, geometric, and topological invariants of a hyperplane arrangement can be expressed in terms of its Tutte polynomial. 
We also show that, even if one is only interested in computing the Tutte polynomial of a graph or a matroid, the theory of hyperplane arrangements provides a powerful finite field method for this computation.
\end{abstract}

\addtocontents{toc}{\protect\setcounter{tocdepth}{2}}
\tableofcontents


%
%
%
%
%
%
%
%
%
%
%
%
%
%
%
%
%
%

\bigskip
\bigskip

\noindent \textsf{Note.} This is a preliminary version of a chapter of the upcoming \emph{CRC Handbook on the Tutte Polynomial and Related Topics}, edited by Joanna A.  Ellis-Monaghan and Iain Moffatt.

\newpage

\section{{Introduction}}\label{f.sec:intro}

The Tutte polynomial is a fundamental invariant associated to a graph, matroid, vector arrangement, or hyperplane arrangement. This short survey focuses on some of the most important results on Tutte polynomials of hyperplane arrangements. We show that many enumerative, algebraic, geometric, and topological invariants of a hyperplane arrangement can be expressed in terms of its Tutte polynomial. We also show that, even if one is only interested in computing the Tutte polynomial of a graph or a matroid, the theory of hyperplane arrangements provides a powerful finite field method for this computation.

%

Our presentation is influenced by a 2002 graduate course on Hyperplane Arrangements by Richard Stanley at MIT, much of which became  the survey \cite{f.Stanleyhyparr}. See \cite{f.OrlikTerao} for a thorough introduction to more algebraic and topological aspects of the theory of hyperplane arrangements.

\section{{Hyperplane Arrangements}}\label{f.sec:hyparrs}


Let ${\mathbbm{k}}$ be a field and $V = {\mathbbm{k}}^d$ be a vector space over ${\mathbbm{k}}$. Let $V^*$ be the dual vector space, which consists of the linear maps or \emph{functionals} from $V$ to $\mathbbm{k}$. 

\begin{definition}A \emph{hyperplane arrangement} ${\mathcal{A}}$ 
\index{hyperplane arrangement}
is a collection of affine hyperplanes in $V$. For each hyperplane $H \in \mathcal{A}$, 
let $l_H \in V^*$ be a nonzero linear functionals and $b_H \in {\mathbbm{k}}$ be a scalar such that 
\[
H = \{ x \in V \, : \, l_H(x) = b_H\}.
\]
\end{definition}

We say ${\mathcal{A}}$ is \emph{central} \index{hyperplane arrangement!central} if all hyperplanes have a common point; in the most natural examples, the origin is a common point. We say it is \emph{essential} \index{hyperplane arrangement!essential} if the intersection of all hyperplanes is a point.
Figure \ref{f.fig:arr}(a)  shows an essential arrangement of 4 hyperplanes in ${\mathbb{R}}^3$.

An important object of study in the theory of hyperplane arrangements is the \emph{complement}  \index{hyperplane arrangement!complement}
$
V({\mathcal{A}}) = V \, \backslash \, \left(\bigcup_{H \in \mathcal{A}} H \right).
$

\begin{figure}[ht]
 \begin{center}
  \includegraphics[scale=.5]{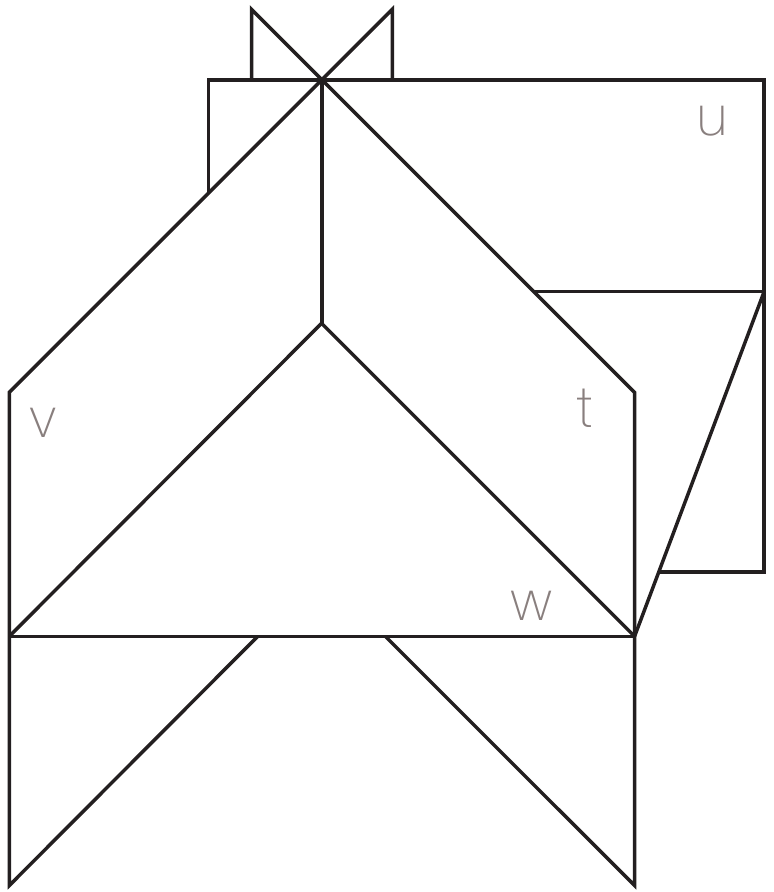}  \qquad \qquad  \qquad \includegraphics[scale=1]{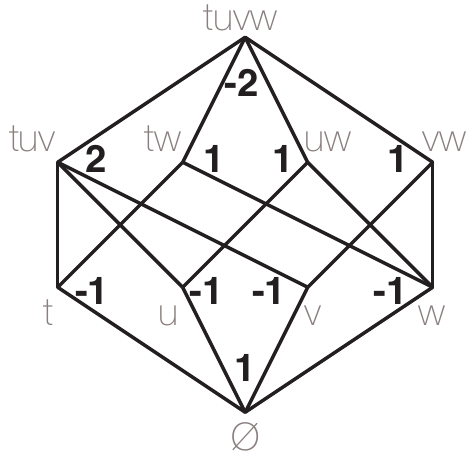}
  \caption{ \label{f.fig:arr}
(a) A hyperplane arrangement $\mathcal{A}$. (b) Its intersection poset $L(\mathcal{A})$ and M\"obius function. Each flat of $\mathcal{A}$ is labeled by the list of hyperplanes containing it.
}
  \end{center}
\end{figure}

\subsection{Intersection posets and matroids}\label{s:matroids}

Define a \emph{flat}  \index{hyperplane arrangement!flat} of ${\mathcal{A}}$ to be an affine subspace obtained as an intersection of hyperplanes in ${\mathcal{A}}$. We often identify a flat $F$ with the set of hyperplanes $\{H_1, \ldots, H_k\}$ of ${\mathcal{A}}$ containing it; clearly we have $F = H_1 \cap \cdots \cap H_k$. 
\begin{definition}
The \emph{intersection poset}  \index{hyperplane arrangement!intersection poset}$L_{\mathcal{A}}$ is the set of flats partially ordered by reverse inclusion of the flats (or inclusion of the sets of hyperplanes containing them). This is a ranked poset, with $r(F) = \dim V - \dim F.$
\end{definition}

If ${\mathcal{A}}$ is central, then $L_{\mathcal{A}}$ is a \emph{geometric lattice} \index{geometric lattice}\cite{f.CrapoRota, f.EC2}. If ${\mathcal{A}}$ is not central, then $L_{\mathcal{A}}$ is only a \emph{geometric meet semilattice} \index{geometric meet semilattice}\cite{f.WachsWalker}. The \emph{rank} $r=r({\mathcal{A}})$ of ${\mathcal{A}}$ is the height of $L_{\mathcal{A}}$.
Figure \ref{f.fig:arr} shows an arrangement and its intersection poset.

Every central hyperplane arrangement has an associated matroid.

\begin{definition}
Let $\mathcal{A}$ be a hyperplane arrangement  in a vector space $V$.
If $\mathcal{A}$ is central, the \emph{matroid}  \index{hyperplane arrangement!matroid} $M(\mathcal{A})$  of $\mathcal{A}$ 
 is the matroid on the ground set $\mathcal{A}$ given by the rank function
\[
r(\mathcal{B}) = \dim V - \dim \bigcap \mathcal{B} \qquad \textrm{ for } \mathcal{B} \subseteq \mathcal{A}.
\]
In general, the \emph{semimatroid} \index{hyperplane arrangement!semimatroid} of $\mathcal{A}$  is the collection of central subsets together with their ranks.
\end{definition}

Semimatroids are equivalent to the \emph{pointed matroids} \index{matroid!pointed} of \cite{f.Brylawski}; see  
\cite{f.Ardilasemimatroids}.

\subsection{Deletion, contraction, centralization, essentialization}

A common technique for inductive arguments is to choose a hyperplane $H$ in an arrangement $\mathcal{A}$  and study how ${\mathcal{A}}$ behaves without $H$ (in the deletion ${\mathcal{A}} \backslash H$) and how $H$ interacts with the rest of ${\mathcal{A}}$ (in the contraction ${\mathcal{A}} / H$).

\begin{figure}[ht]
 \begin{center}
  \includegraphics[scale=.5]{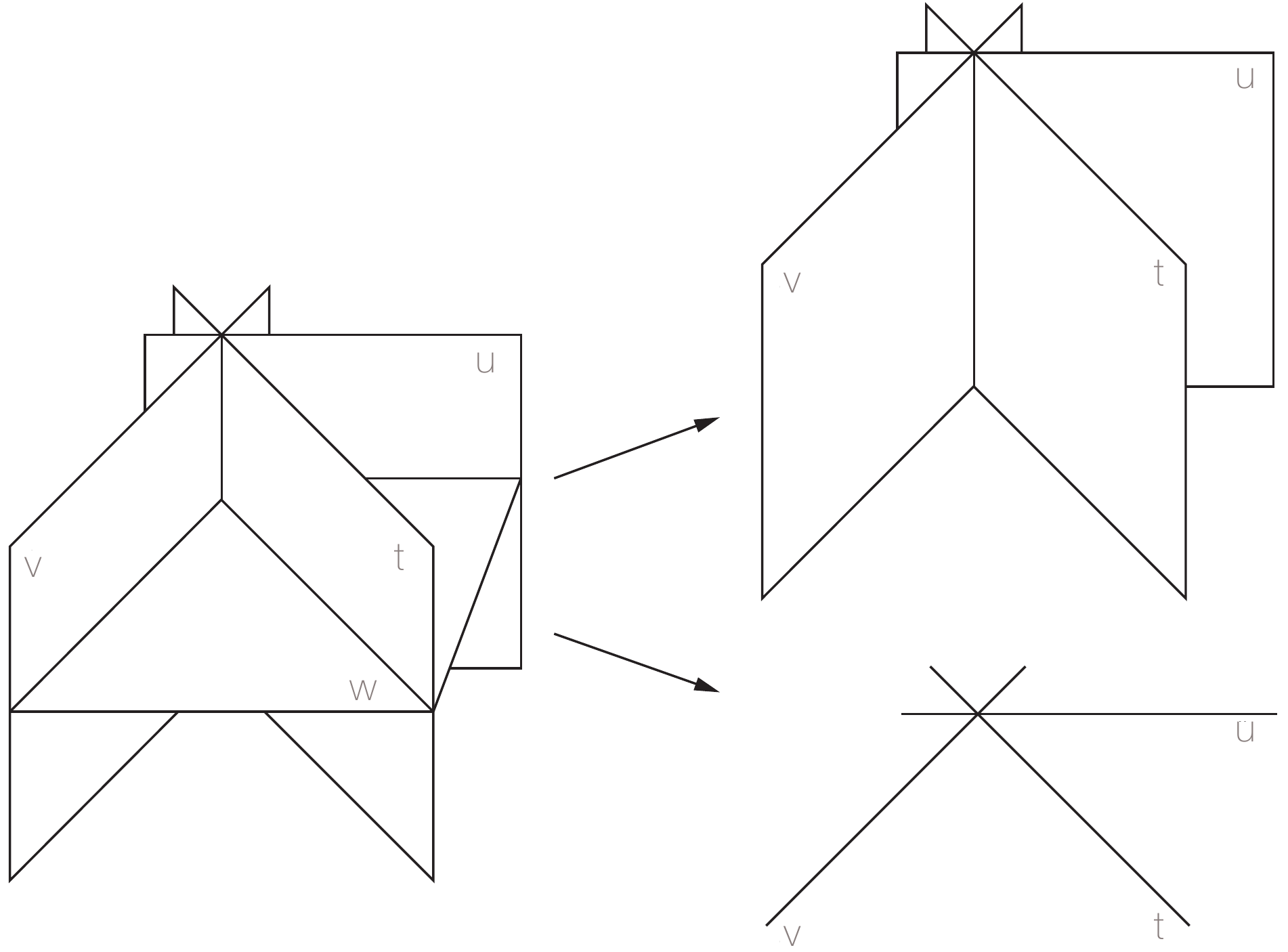}
  \caption{ \label{f.fig:deletioncontraction}
  An arrangement ${\mathcal{A}}$ and its deletion ${\mathcal{A}} \backslash w$ (top) and contraction ${\mathcal{A}}/w$ (bottom).}
 \end{center}
\end{figure}

\begin{definition}
For a hyperplane $H$ of an 
Let ${\mathcal{A}}$ be an arrangement in $V$ and let $H$ be a hyperplane in $\mathcal{A}$.
\begin{enumerate}
\item
The  \emph{deletion}  \index{hyperplane arrangement!deletion}
\[
{\mathcal{A}} \backslash H = \{A \in {\mathcal{A}} \, : \, A \neq H\} \\
\]
is the arrangement in the same ambient space $V$ consisting of the hyperplanes other than $H$.
\item
The \emph{contraction}  \index{hyperplane arrangement!contraction}
\[
{\mathcal{A}} / H = \{A \cap H \, : \, A \in {\mathcal{A}}, A \neq H\}
\]
is the arrangement in the new ambient space $H$ consisting of the intersections of the other hyperplanes with $H$.
\end{enumerate}
\end{definition}

 Figure \ref{f.fig:deletioncontraction} shows an arrangement ${\mathcal{A}} = \{t,u,v,w\}$ in ${\mathbb{R}}^3$ together with the deletion ${\mathcal{A}} \backslash w$ and contraction ${\mathcal{A}}/w$.

\begin{remark}\label{f.remark}
It is somewhat inconvenient that hyperplane arrangements are not closed under contraction. For example, in Figure \ref{f.fig:deletioncontraction}, the image of
$t$ in $({\mathcal{A}}/u)/v$
is not a hyperplane, but the whole ambient space. We will circumvent this difficulty by considering arrangements where the full-dimensional ambient space is allowed as a degenerate ``hyperplane". However, when we make statements about the complement $V(\mathcal{A})$, we will assume that $\mathcal{A}$ does not contain the degenerate hyperplane.

A more robust solution is to work more generally in the context of  \emph{matroids} \cite{f.CrapoRota, f.Oxley} for central arrangements, and  \emph{pointed matroids} \cite{f.Brylawski} or \emph{semimatroids} \cite{f.Ardilasemimatroids} for affine arrangements. However, to keep the presentation short and self-contained, we will not pursue this point of view.
\end{remark}

We say a hyperplane $H$ of an arrangement $\mathcal{A}$ in a vector space $V$ is a \emph{loop} \index{hyperplane arrangement!loop} if it is the degenerate hyperplane $H=V$. We say it is a \emph{coloop} \index{hyperplane arrangement!coloop} if it intersects the rest of the arrangement transversally; that is, if $r({\mathcal{A}}) = r({\mathcal{A}} \backslash H) + 1$. For example, $w$ is a coloop in the arrangement of Figure \ref{f.fig:deletioncontraction}.

In some ways central arrangements are slightly better behaved than affine arrangements. We can \emph{centralize} an affine arrangement ${\mathcal{A}}$ as follows.

\begin{definition} The \emph{centralization} \index{hyperplane arrangement!centralization} or \emph{cone} \index{hyperplane arrangement!cone} of a hyperplane arrangement $\mathcal{A}$ in $\mathbbm{k}^d$ is the arrangement $c{\mathcal{A}}$ in ${\mathbbm{k}}^{d+1}$ obtained by converting each hyperplane $a_1x_1 + \cdots + a_dx_d = a$ in ${\mathbbm{k}}^d$ into the hyperplane $a_1x_1 + \cdots + a_dx_d = ax_{d+1}$ in ${\mathbbm{k}}^{d+1}$, and adding the new hyperplane $x_{d+1}=0$.
\end{definition}

Sometimes arrangements are ``too central", in the sense that their intersection is a subspace $L \subset V$ of positive dimension. In that case, there is little harm in modding out the arrangement by $L$, as follows.

\begin{definition}
The \emph{essentialization} \index{hyperplane arrangement!essentialization} of  a central arrangement ${\mathcal{A}}$ in $V$ is the arrangement 
ess$({\mathcal{A}}) = \{H / L  \, : \, H \in {\mathcal{A}}\}$ in the quotient vector space $V/L$. 
\end{definition}

The resulting arrangement  is essential. In most situations of interest, there is no important difference between ${\mathcal{A}}$ and ess$({\mathcal{A}})$.

\section{{Polynomial Invariants}}\label{f.sec:invariants}

Different choices of the ground field $\mathbbm{k}$  lead to different questions about arrangements ${\mathcal{A}}$ and their complements $V({\mathcal{A}})$. In many of these questions, a crucial role is played by two combinatorial polynomials which we now define.

\subsection{{The characteristic and Tutte polynomials}} \label{f.sec:charpoly}


\begin{definition}
The \emph{M\"obius function} \index{hyperplane arrangement!M\"obius function} $\mu: L(\mathcal{A}) \rightarrow \mathbb{Z}$ of (the intersection poset of) an arrangement $\mathcal{A}$ is defined recursively by decreeing that for every flat $G \in L(\mathcal{A})$,
\begin{equation}\label{f.e:Mobius1}
\sum_{F \leq G} \mu(F) = \begin{cases}
1 & \textrm{ if } G \textrm{ is the minimum element of } L(\mathcal{A}), \\
0 & \textrm{ otherwise}.
\end{cases}
\end{equation}
The \emph{characteristic polynomial} \index{hyperplane arrangement!characteristic polynomial} of ${\mathcal{A}}$ is 
\[
\chi({\mathcal{A}};q) = \sum_{F \in L_{\mathcal{A}}} \mu(F) q^{\dim F}.
\]
\end{definition}

\begin{example}
For the arrangement of Figure \ref{f.fig:arr}, the M\"obius function is shown in dark labels next to the intersection poset. The coefficients of the characteristic polynomial $\chi({\mathcal{A}};q) = q^3 - 4q^2 + 5q - 2$ is easily computed by adding the M\"obius numbers on each level of $L_{\mathcal{A}}$.
\end{example}



%
%


\begin{definition}
The \emph{Tutte polynomial} \index{hyperplane arrangement!Tutte polynomial} of an arrangement $\mathcal{A}$ in a vector space $V$ is
\begin{equation}\label{f.th:Tutteformula}
T({\mathcal{A}};x,y) = \sum_{{\mathcal{B}} \subseteq {\mathcal{A}} \atop {\mathcal{B} \textrm{ central}}} (x-1)^{r-r({\mathcal{B}})} \, (y-1)^{|{\mathcal{B}}|-r({\mathcal{B}})},
\end{equation}
where the sum is taken over all the central subarrangements ${\mathcal{B}}$ of ${\mathcal{A}}$,
and we write $r({\mathcal{B}}) = \dim V - \dim \bigcap {\mathcal{B}}$ and $r=r({\mathcal{A}})$. 
\end{definition}

The Tutte polynomial was defined for graphs, matroids, and arrangements in \cite{f.Tuttecontribution}, \cite{f.Crapo}, and \cite{f.ArdilaTutte} respectively. When $\mathcal{A}$ is central, the above definition coincides with the usual matroid-theoretic definition.

\begin{example}\label{f.ex:Tutte}
For the arrangement of Figure \ref{f.fig:arr}, (\ref{f.th:Tutteformula}) yields
\[
T(\mathcal{A};x,y) = (x-1)^3+4(x-1)^2+6(x-1)+3+(x-1)(y-1)+(y-1) = x^3+x^2+xy.
\]
\end{example}

The large amount of cancellation in the computation above is systematically explained by the following theorem. Let us fix a linear order on $\mathcal{A}$, and let $\mathcal{C}_{>H} = \{C \in \mathcal{C} \, : \, C > H\}$  for any subarrangement $\mathcal{C} \subseteq \mathcal{A}$ and hyperplane $H \in \mathcal{A}$.  We define a \emph{basis} of $\mathcal{A}$ to be a central subset of maximal rank $r$. 

Let $\mathcal{B}$ be a basis. Say a hyperplane $H \notin \mathcal{B}$ is \emph{externally active} \index{hyperplane arrangement!external activity} with respect to $\mathcal{B}$ if $\mathcal{B} \cup H$ is central and $r(\mathcal{B}_{>H} \cup H) = r(\mathcal{B}_{>H})$.
Say $H \in \mathcal{B}$ is \emph{internally active} \index{hyperplane arrangement!internal activity} if $r((\mathcal{B}-H) \cup \mathcal{A}_{<H}) = r(\mathcal{B}-H) = r-1$. Let $e(\mathcal{B})$ and $i(\mathcal{B})$ be the number of externally and internally active elements with respect to $\mathcal{B}$, respectively.

\begin{theorem} \cite{f.Ardilasemimatroids}
For any linear order on the hyperplanes of an arrangement $\mathcal{A}$, the Tutte polynomial of  $\mathcal{A}$ is given by 
$T(\mathcal{A};x,y) = \displaystyle  \sum_{\mathcal{B} \textrm{ basis}} x^{i(\mathcal{B})}y^{e(\mathcal{B})}.
$
\end{theorem}

In Example \ref{f.ex:Tutte}, the three monomials of $T(\mathcal{A};x,y)$ correspond to the bases $tuw, tvw, uvw$ of $\mathcal{A}$. We invite the reader to choose a linear order for $\mathcal{A}=\{t,u,v,w\}$ and verify that these bases give the monomials $x^3, x^2$, and $xy$.

Although it is not obvious from its definition, the characteristic polynomial is a specialization of the Tutte polynomial.

\begin{theorem} (Whitney's Theorem) \cite{f.ArdilaTutte,f.Whitney}
\label{f.th:charpoly}
The characteristic polynomial and the Tutte polynomial of an arrangement of rank $r$ in $\mathbbm{k}^d$ are related by 
\[
\chi({\mathcal{A}};q) = (-1)^r q^{d-r} T({\mathcal{A}};1-q,0).
\]
\end{theorem}

This is part of a general phenomenon that we explore in the next section.

\subsection{Tutte-Grothendieck invariants, recursion, universality}

As evidenced by this Handbook, the Tutte polynomial appears naturally in numerous different contexts, and provides the answer to many enumerative, algebraic,  topological, and geometric questions. This is certainly true in the context of  hyperplane arrangements; when we encounter a new quantity or polynomial associated to an arrangement, a good first question to ask is whether it is an evaluation of the Tutte polynomial.

The ubiquity of the Tutte polynomial is not accidental: this polynomial is universal among a large, important family of invariants of hyperplane arrangements, as we now make precise. Let $R$ be a ring, and let $\textrm{HypArr}$ be the collection of all hyperplane arrangements over a field ${{\mathbbm{k}}}$. As explained in Remark \ref{f.remark}, we need to allow our arrangements to contain the ambient space as a degenerate hyperplane.

\begin{definition}
A function $f:\textrm{HypArr} \rightarrow R$ is a \emph{generalized Tutte-Grothendieck invariant} \index{hyperplane arrangement!Tutte-Grothendieck invariant} if $f(\mathcal{A}_1) = f(\mathcal{A}_2)$ for any arrangements $\mathcal{A}_1$ and $\mathcal{A}_2$ with isomorphic semimatroids, and if 
for every arrangement $\mathcal{A}$ and every hyperplane $H \in \mathcal{A}$, we have
\begin{equation} \label{f.eq:T-G}
f({\mathcal{A}}) =
\begin{cases} 
a f({\mathcal{A}} \backslash H) +  b f({\mathcal{A}} / H)  & \textrm{ if $H$ is neither a loop nor a coloop} \\
f({\mathcal{A}} \backslash H) f(L)  &  \textrm{ if $H$ is a loop}\\
f({\mathcal{A}} / H) f(C)  & \textrm{ if $H$ is a coloop}
\end{cases}
\end{equation}
for some non-zero constants $a,b \in R$. Here $f(L)$ and $f(C)$ denote the (necessarily well-defined) function of a single loop $L$ and a single coloop $C$, respectively. We say $f({\mathcal{A}})$  is a \emph{Tutte-Grothendieck invariant} when $a=b=1$.
\end{definition}

\begin{theorem} \cite{f.ArdilaTutte, f.Ardilasemimatroids, f.Crapo, f.Tuttecontribution} \label{f.th:Tutterecursion} 
The Tutte polynomial is a universal Tutte-Grothendieck invariant for $\mathrm{HypArr}$, namely,
\begin{enumerate}
\item 
The Tutte polynomial $T({\mathcal{A}};x,y)$ satisfies (\ref{f.eq:T-G}) with $a=b=1$, $f(C)=x$, and $f(L)=y$.
\item Any generalized Tutte-Grothendieck invariant is a function of 
 the Tutte polynomial. Explicitly, if $f$ satisfies (\ref{f.eq:T-G}), then
\[
f({\mathcal{A}}) = a^{n-r} \, b^r  \, T\left({\mathcal{A}};\frac{f(C)}{b}, \frac{f(L)}{a}\right).
\]
where $n$ is the number of elements and $r$ is the rank  of ${\mathcal{A}}$.\end{enumerate}
\end{theorem}

Part 1 of this theorem implies that the Tutte polynomial can also be defined alternatively by the recursion (\ref{f.eq:T-G}) with $a=b=1$, $f(C)=x$, and $f(L)=y$.
In Part 2 we do not need to assume that $a$ and $b$ are invertible; when we multiply by $a^{n-r} \, b^r$, all denominators cancel out.

\section{{{Topological and Algebraic Invariants }}} \label{f.sec:applications}

As is the case with graphs and matroids, many important invariants of a hyperplane arrangement are generalized Tutte-Grothendieck invariants, and hence are evaluations of the Tutte polynomial. In this section we collect, without proofs, a few selected results of this flavor.

\subsection{{{Topological invariants of arrangements}}}

\begin{theorem}\label{f.th:charpoly} The characteristic polynomial $\chi({\mathcal{A}};x)$ contains the following information about the complement $V({\mathcal{A}})$ of a hyperplane arrangement ${\mathcal{A}}$.

\begin{enumerate}
\item
 $({\mathbbm{k}} = {\mathbb{R}})$
\cite{f.Zaslavsky} 
Let ${\mathcal{A}}$ be a hyperplane arrangement in ${\mathbb{R}}^d$. Let ${\mathcal{A}}$ separate the complement $V(\mathcal{A})$ into $a(\mathcal{A})$ connected components or \emph{regions}\index{hyperplane arrangement!region}. Let $b(\mathcal{A})$ be the number of  bounded regions \index{hyperplane arrangement!region!bounded} of the essentialization $\mathrm{ess}(\mathcal{A})$. Then
\[
a({\mathcal{A}}) = (-1)^d\chi({\mathcal{A}};-1), \qquad b({\mathcal{A}}) = (-1)^{r({\mathcal{A}})}\chi({\mathcal{A}};1).
\]

\item
$({\mathbbm{k}} = {\mathbb{C}})$
\cite{f.GoreskyMacPherson, f.OrlikSolomon}
Let ${\mathcal{A}}$ be a hyperplane arrangement in ${\mathbb{C}}^d$. The integral cohomology ring of the complement $V({\mathcal{A}})$ has Poincar\'e polynomial
\[
\sum_{k \geq 0} \mathrm{rank } \, H^k(V({\mathcal{A}}), {\mathbb{Z}}) q^k = (-q)^d\chi\left({\mathcal{A}};\frac{-1}{q} \right).
\]

\item
$({\mathbbm{k}} = {\mathbb{F}}_q)$
\cite{f.Athanasiadis, f.CrapoRota}
Let ${\mathcal{A}}$ be a hyperplane arrangement in ${\mathbb{F}}_q^d$ where ${\mathbb{F}}_q$ is the finite field of $q$ elements. The complement $V({\mathcal{A}})$ has size
\[
| V({\mathcal{A}}) | = \chi({\mathcal{A}};q).
\]
\end{enumerate}
\end{theorem}

\begin{theorem} \cite{f.BrylawskiOxley}
Let ${\mathcal{A}}$ be a central arrangement in ${{\mathbbm{R}}}^d$. 

\begin{enumerate}
\item
 Consider an affine hyperplane $H$ which is in general position with respect to ${\mathcal{A}}$. Then the number of regions of ${\mathcal{A}}$ which  have a bounded (and non-empty) intersection with $H$ equals $T(\mathcal{A}; 1,0)$, the absolute value of the last coefficient of $\chi({\mathcal{A}};q)$. In particular, this number is independent of $H$.
\item
 Add to ${\mathcal{A}}$ an affine hyperplane $H'$ which is a parallel translation of one of the hyerplanes $H \in {\mathcal{A}}$. The number of bounded regions of ${\mathcal{A}} \cup H'$ is the \emph{beta invariant} \index{hyperplane arrangement!beta invariant} of ${\mathcal{A}}$, which is the coefficient of $x^1y^0$ and of $x^0y^1$ in $T({\mathcal{A}};x,y)$.
 In particular, this number is independent of $H$.
\end{enumerate}
\end{theorem}

One very important algebraic topological invariant of a complex arrangement $\mathcal{A}$ is the cohomology ring $H^*(V(\mathcal{A}), {\mathbb{Z}})$ of its complement, known as the \emph{Orlik-Solomon algebra} \index{hyperplane arrangement!Orlik-Solomon algebra} of $\mathcal{A}$. 
It has the following combinatorial presentation.

\begin{theorem} \cite{f.OrlikSolomon} Let $\mathcal{A}$ be a central arrangement in $\mathbb{C}^d$.
Let $E$ be the exterior algebra with generators $e_H$ for each $H \in \mathcal{A}$. For each ordered set of hyperplanes $S=\{H_1, \ldots, H_k\}$ let $e_S = e_{H_1} \wedge \cdots \wedge e_{H_k}$ and let $\partial e_S = \sum_{j=1}^k (-1)^{j-1} e_{S - H_j}$. Say that $S$ is \emph{dependent} if $\dim \cap S > d - |S|$, or equivalently, if $l_{H_1}, \ldots, l_{H_k}$ are linearly dependent. Then 
\[
H^*(V(\mathcal{A}), {\mathbb{Z}}) \cong E / \left\langle \, \partial e_S \, : \, S \subseteq \mathcal{ A} \textrm{ is dependent} \,  \right\rangle.
\]
\end{theorem}

Another important invariant of a complex arrangement $\mathcal{A}$ is the cohomology ring $H^*(W(\mathcal{A}), {\mathbb{Z}})$ of the  \emph{wonderful compactification} \index{hyperplane arrangement!wonderful compactification} $W(\mathcal{A})$ of the complement $V(\mathcal{A})$ constructed by De Concini and Procesi \cite{f.DeConciniProcesiwonderful}. It
also has an elegant combinatorial presentation.

\begin{theorem} \cite{f.FeichtnerYuzvinsky} 
Let $\mathcal{A}$ be a central arrangement in $\mathbb{C}^d$ and  $W(\mathcal{A})$ be the maximal wonderful compactification of its complement. 
Then 
\begin{equation}\label{f.e:Chow}
H^*(W(\mathcal{A}), {\mathbb{Z}}) \cong S_{\mathcal{A}} / (I_{\mathcal{A}} + J_{\mathcal{A}})
\end{equation}
where 
\begin{align*}
&
S_{\mathcal{A}} = \mathbb{Z}[\,x_F \, : \, \emptyset \subsetneq F \subsetneq \mathcal{A} \textrm{ is a flat of } \mathcal{A}\,], \\
&
I_{\mathcal{A}} = \langle x_{F_1}x_{F_2} \, : \, F_1, F_2 \textrm{ are incomparable proper flats} \rangle, \\
&
J_{\mathcal{A}} = \left\langle \sum_{F \ni i} x_F - \sum_{F \ni j} x_F , : \, i \neq j \textrm { in } \mathcal{A} \right\rangle.
\end{align*}
\end{theorem}

One may use (\ref{f.e:Chow}) as the definition of the \emph{Chow ring} of any matroid. This ring is one of the crucial ingredients in the solution to the following central problem in matroid theory,  conjectured by Rota in 1970 \cite{f.RotaICM}.

\begin{theorem}  \cite{f.AdiprasitoHuhKatz, f.Huh, f.HuhKatz}
The characteristic polynomial of an arrangement
\[
\chi_{\mathcal{A}}(q) = q^n - a_{n-1}q^{n-1} + a_{n-2}q^{n-2} - \cdots + (-1)^n a_n q^0
\]
has coefficients alternating signs, so $a_i \geq 0$. Furthermore, the coefficients are unimodal and even log-concave, that is,
\begin{align*}
a_1 \leq a_2 \leq \cdots \leq a_{i-1} \leq a_i \geq a_{i+1} \geq \cdots \geq a_n \qquad & \textrm{for some  $i$, and} \\
a_{j-1}a_{j+1} \leq a_j^2 \qquad & \textrm{for all }j.
\end{align*}

\end{theorem}

Rota's conjecture was proved very recently by Huh \cite{f.Huh} for arrangements over fields of characteristic $0$, by Huh and Katz \cite{f.HuhKatz} for arrangements over arbitrary fields, and by Adiprasito, Huh, and Katz for arbitrary matroids \cite{f.AdiprasitoHuhKatz}. 


\subsection{{{Algebras from arrangements}}} There are other natural algebras related to the Tutte polynomial of an arrangement arising in commutative algebra, hyperplane arrangements, box splines, and index theory; we discuss a few. 

Throughout this section we assume our arrangement ${\mathcal{A}}$ is central. 
For each hyperplane $H $ in a hyperplane arrangement ${\mathcal{A}}$ in ${{\mathbbm{k}}}^d$ let $l_H$ be a linear functional in $({{\mathbbm{k}}}^d)^*$ such that $H$ is given by the equation $l_H(x)=0$.

Our first example is a family of graded vector spaces $C_{{\mathcal{A}},k}$ associated to an arrangement $\mathcal{A}$. For $k=0,-1,-2$, they 
arose in the theory of splines \cite{f.DahmenMicchelli, f.DeConciniProcesibook, f.HoltzRon}  as the spaces of solutions to certain systems of differential equations.

\newpage

\begin{theorem} Let $\mathcal{A}$ be an arrangement over a field of characteristic zero.
 \begin{enumerate}
\item \cite{f.Wagner} 
Let $C_{{\mathcal{A}},0} = {\mathrm{span }} \{\prod_{H \in {\mathcal{B}}} l_H \, : \, {\mathcal{B}} \subseteq {\mathcal{A}}\}$. This is a subspace of a polynomial ring in $d$ variables, graded by degree. Its dimension is $T(2,1)$ and its Hilbert polynomial is 
\[
\textrm{Hilb}(C_{{\mathcal{A}}, 0}; q) = \sum_{j \geq 0} \dim (C_{{\mathcal{A}},0})_j \, q^j = 
q^{d-r}T\left(1+q, \frac1q\right).
\]

\item \cite{f.Ardilathesis, f.ArdilaPostnikov, f.DahmenMicchelli, f.DeConciniProcesibook, f.HoltzRon, f.PostnikovShapiro, f.PostnikovShapiroShapiro} More generally, let $C_{{\mathcal{A}},k}$ be the vector space of polynomial functions such that the restriction of $f$ to any line $h$ has degree at most $\rho_{{\mathcal{A}}}(h)+k$, where $\rho_{{\mathcal{A}}}(h)$ is the number of hyperplanes of ${\mathcal{A}}$ not containing $h$. It is not obvious, but true, that this definition of $C_{{\mathcal{A}},0}$ matches the one above. 
We have
\[
\textrm{Hilb}(C_{{\mathcal{A}}, -1}; q) = q^{d-r}T\left(1, \frac1q\right), \qquad \textrm{Hilb}(C_{{\mathcal{A}}, -2}; q) = q^{d-r}T\left(0, \frac1q\right)
\]
and similar formulas hold for any $k \geq -2$. 
\end{enumerate}
\end{theorem}

Another example that arises in several contexts is  the following.

\begin{theorem}
\cite{f.BrionVergne, f.ProudfootSpeyer, f.Teraoalgebras} Let $R({\mathcal{A}})$ be the vector space of rational functions whose poles are in ${\mathcal{A}}$. It is the ${\mathbbm{k}}$-algebra of rational functions generated by $\{1/l_H \, : \, H \in {\mathcal{A}}\}$, and we grade it so that $\deg(1/l_H)=1$. Then
\[
\textrm{Hilb}(R({\mathcal{A}}); q) = \frac{q^d}{(1-q)^d} T\left( \frac1q, 0 \right).
\]
\end{theorem}

\section{{{The Finite Field Method}}} 
\label{f.sec:finitefield} 

Even if one is primarily interested in fields of characteristic zero, it is also quite useful to consider hyperplane arrangements over the finite field $\mathbb{F}_q$ of $q$ elements, where $q$ is a prime power. The following variant of the Tutte polynomial plays an important role.

\begin{definition}
The \emph{coboundary polynomial} \index{hyperplane arrangement!coboundary polynomial} $\overline{\chi}({\mathcal{A}};X,Y)$ is the following simple transformation of the Tutte polynomial:
\begin{equation}\label{f.eq:coboundary}
\overline{\chi}({\mathcal{A}};X,Y) = (Y-1)^rT\left({\mathcal{A}} ; \frac{X+Y-1}{Y-1}, Y\right). 
\end{equation}
A simple change of variables allows us to recover $T({\mathcal{A}};x,y)$ from $\overline{\chi}({\mathcal{A}};X,Y)$.
\end{definition}

Let $\mathcal{A}$ be an arrangement over a field of characteristic zero. Say $\mathcal{A}$ is a \emph{$\mathbb{Q}$-arrangement} if its defining equations have rational coefficients. For any power $q$ of a large enough prime, the equations of $\mathcal{A}$ also define a hyperplane arrangement $\mathcal{A}_q$ over $\mathbb{F}_q$. Say that $\mathcal{A}$ \emph{reduces correctly} over $\mathbb{F}_q$ if the intersection posets of $\mathcal{A}$ and $\mathcal{A}_q$ are isomorphic,  and hence $\mathcal{A}$ and $\mathcal{A}_q$ have the same Tutte polynomial.

\newpage

\begin{theorem}\label{f.th:Tuttefinitefield} (Finite Field Method) \index{hyperplane arrangement!finite field method} \cite{f.ArdilaTutte, f.CrapoRota, f.Greene, f.WelshWhittle} 
\begin{enumerate}
\item
Let ${\mathcal{A}}$ be a hyperplane arrangement of rank $r$ in ${\mathbb{F}}_q^d$. For each point $p \in {\mathbb{F}}_q^d$ let $h(p)$ be the number of hyperplanes of ${\mathcal{A}}$ containing $p$. Then
\begin{equation}\label{f.e:finfield}
\sum_{p \in {\mathbb{F}}_q^d} t^{h(p)} = q^{d-r} \overline{\chi}({\mathcal{A}};q,t).
\end{equation}
\item
Let  ${\mathcal{A}}$ be a $\mathbb{Q}$-arrangement over a field of characteristic zero. For any power $q$ of a large enough prime, $\mathcal{A}$ reduces correctly over $\mathbb{F}_q$, and the coboundary polynomial $\overline{\chi}({\mathcal{A}};q,t) = \overline{\chi}({\mathcal{A}_q};q,t)$ of $\mathcal{A}$ may be computed using (\ref{f.e:finfield}).
\end{enumerate}
\end{theorem}

Theorem \ref{f.th:Tuttefinitefield} is one of the most effective methods for computing Tutte polynomials of a hyperplane arrangement ${\mathcal{A}}$.
It reduces the computation of $T({\mathcal{A}};x,y)$ to an enumerative problem over finite fields, which can sometimes be solved \cite{ f.ArdilaTutte, f.Athanasiadis}.
This method also works for any graph or any matroid realizable over $\mathbb{Q}$, since they can be regarded as hyperplane arrangements as well.
Let us illustrate this with two simple examples. 

\begin{example}
If we think of the arrangement of Figure \ref{f.fig:arr} as living inside the ambient space $\mathbb{F}_q^3$, a careful enumeration gives
\[
\overline{\chi}(\mathcal{A};q,t) = t^4 + (q-1)t^3 + 3(q-1)t^2 + (4q^2-9q+5)t+(q^3-4q^2+5q-2),
\]
in agreement with Example \ref{f.ex:Tutte}.
\end{example}

\begin{proposition}
For the \emph{coordinate arrangement} $\mathcal{H}_n$ consisting of the $n$ coordinate hyperplanes in ${\mathbbm{C}}^n$,
\[
\overline{\chi}({\mathcal{H}_n};X,Y) = (X+Y-1)^n.
\]
\end{proposition}

\begin{proof}
Regard $\mathcal{H}_n$ as an arrangement over $\mathbb{F}_q$ for a power $q$ of a large prime. By (\ref{f.e:finfield}), we need to count the  points in $\mathbb{F}_q^n$ that are on exactly $k$ of the $n$ hyperplanes $x_i=0$ for $1 \leq i \leq n$. To choose such a point, we may first choose which $k$ hyperplanes it is on, and then choose its remaining $n-k$ non-zero coordinates independently, for a total of ${n \choose k} (q-1)^{n-k}$ choices.
It follows that 
\[
\overline{\chi}({\mathcal{H}_n};q,t) = 
\sum_{p \in {\mathbb{F}}_q^n} t^{h(p)} =
\sum_{k=0}^n {n \choose k} (q-1)^{n-k}t^k = (q+t-1)^n
\]
as desired.
\end{proof}

\section{{A Catalog of Characteristic and Tutte Polynomials}}\label{f.sec:computing} 

Computing Tutte polynomials is extremely difficult in general, as explained in 
\cite{f.Welshcomplexity}. However, the computation is  possible in some cases. We now survey some of the most interesting examples; see \cite{f.MerinoRamirezetal} for others. Some of these formulas are best expressed in terms of the coboundary polynomial $\overline{\chi}({\mathcal{A}};X,Y)$, which is equivalent to the Tutte polynomial $T({\mathcal{A}};x,y)$ by (\ref{f.eq:coboundary}).
Almost all of them are most easily proved using the finite field method of Theorems \ref{f.th:charpoly}.3 and \ref{f.th:Tuttefinitefield}.

\begin{enumerate}


\item \cite{f.BrylawskiOxley}
If the characteristic of $\mathbbm{k}$ is $0$, any \emph{sufficiently generic} \index{hyperplane arrangement!generic} central arrangement in $\mathbbm{k}^d$ is such that the intersection of any $m$ hyperplanes has codimension $m$ for $1 \leq m \leq d$. If we let ${\mathcal{A}}_{n,d}$ be such a \emph{generic arrangement} of $n$ hyperplanes in $\mathbbm{k}^d$, then 
\[
T({{\mathcal{A}}_{n,d}};x,y) = \sum_{i=1}^d{n - i - 1 \choose n-d-1} x^i +  \sum_{j=1}^{n-d} {n-j-1 \choose d-1} y^j.
\]

\item \cite{f.BrylawskiOxley}
A graph $G$ on $n$ vertices gives rise to the \emph{graphical arrangement} \index{hyperplane arrangement!graphical} ${\mathcal{A}}_{G}$ in ${\mathbbm{k}}^n$ which has a hyperplane $x_i = x_j$ for every edge $ij$ of $G$. If ${\mathbbm{k}} = \mathbb{R}$, the regions of ${\mathcal{A}}_G$ are in bijection with  the orientations of the edges of $G$ that form no directed cycles.

By the finite field method, the characteristic polynomial $\chi({{\mathcal{A}}_G;q)}$ is equal to the \emph{chromatic polynomial} \index{chromatic polynomial} $\chi(G; q)$, which counts the vertex colorings of $G$ with $q$ colors such that no edge joins two vertices of the same color. This gives a proof that $\chi(G;q)$ is indeed polynomial in $q$. 
Similarly, 
\[
q^{n-r} \overline{\chi}({{\mathcal{A}}_G}; q,t) =
\sum_{f: [n] \rightarrow [q]} t^{h(f)} 
\]
where we sum over all vertex colorings $f$ of $G$ with $q$ colors, and $h(f)$ is the number of edges of $G$ whose ends have the same color in $f$.

An important special case is the graphical arrangement for the complete graph $K_n$, consisting of the ${n \choose 2}$ hyperplanes in $\mathbb{R}^n$ given by equations $x_i=x_j$ for $1 \leq i < j \leq n$. This is known as the \emph{braid arrangement} \index{hyperplane arrangement!braid} or the \emph{type $A$ Coxeter arrangement}, and we now discuss it further.

\item
 \cite{f.ArdilaTutte, f.Tuttedichromatic} Root systems are arguably the most important vector configurations; these highly symmetric arrangements are fundamental in many branches of mathematics. For the definition and properties, see for example \cite{f.Humphreys}; we focus on the four infinite families of \emph{classical root systems}: \index{root system}
\begin{eqnarray*}
A_{n-1} &=& \{e_i-e_j,\, : \, 1\leq i < j\leq n\} \\
B_n &=& \{e_i -  e_j, e_i + e_j \, : \,  1\leq i <  j\leq n\} \cup \{e_i \, : \, 1 \leq i \leq n\} \\
C_n &=& \{e_i -  e_j, e_i + e_j \, : \,  1\leq i <  j\leq n\} \cup \{2e_i \, : \, 1 \leq i \leq n\} \\
D_n &=& \{e_i -  e_j, e_i + e_j \, : \,  1\leq i <  j\leq n\} 
\end{eqnarray*}
regarded as linear functionals  in $(\mathbbm{k}^n)^*$, where $e_1, \ldots, e_n$ is the standard basis. Figure \ref{f.fig:rootsystems} illustrates the two-dimensional examples. Aside from the infinite families, there are five exceptional root systems: $E_6, E_7, E_8, F_4, G_2$.

\begin{figure}[ht]
 \begin{center}
  \includegraphics[scale=.4]{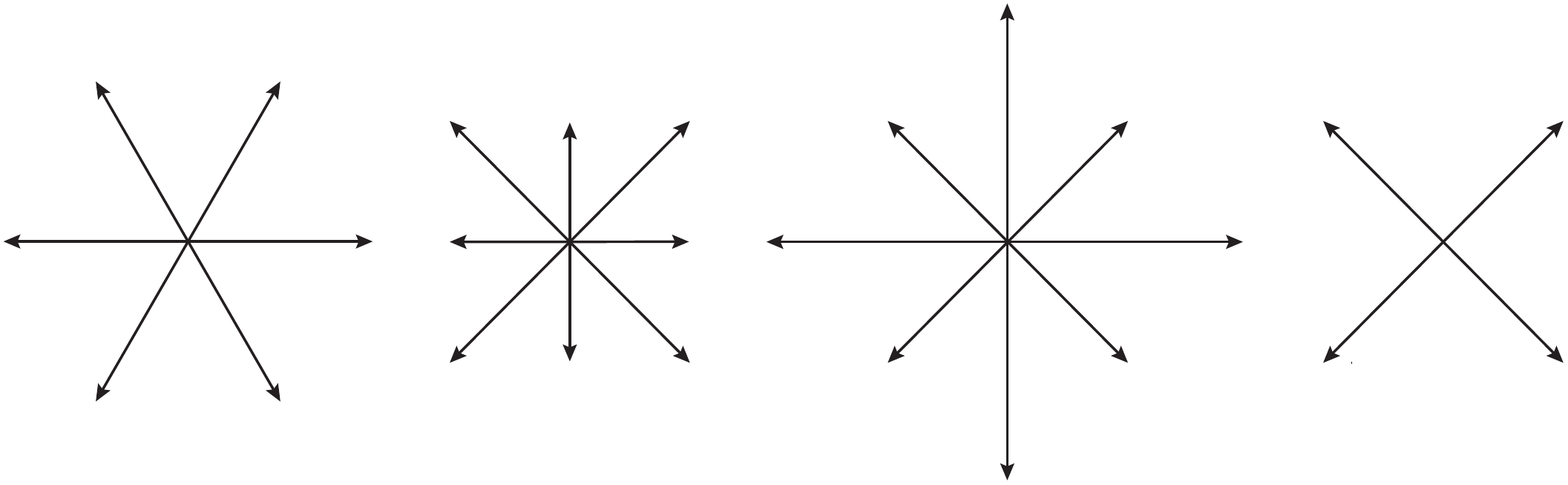}
  \caption{ \label{f.fig:rootsystems}
The root systems $A_2, B_2, C_2,$ and $D_2$, respectively.} \end{center}
\end{figure}

The classical root systems lead to the \emph{Coxeter arrangements} \index{hyperplane arrangement!Coxeter} ${\mathcal{A}}_{n-1}, \mathcal{BC}_n,$ and $\mathcal{D}_n$ of hyperplanes determined by the roots. For example, the Coxeter arrangement ${\mathcal{A}}_{n-1}$ is the braid arrangement in $\mathbb{R}^n$. Note that $B_n$ and $C_n$  lead to the same arrangement $\mathcal{BC}_n$.

The characteristic polynomials of Coxeter arrangements are very elegant:
\begin{eqnarray*}
\overline{\chi}({A_{n-1}};q) &=& q(q-1)(q-2) \cdots (q-n+1),\\
\overline{\chi}({BC_n};q) &=& (q-1)(q-3)\cdots (q-2n+3)(q-2n+1),\\
\overline{\chi}({D_n};q) &=& (q-1)(q-3)\cdots (q-2n+3)(q-n+1).
\end{eqnarray*}
Similar expressions hold for the exceptional root systems. There are several conceptual explanations for the factorization of these polynomials into linear forms; see \cite[Section 1.7.4]{f.Ardilasurvey} or  \cite{f.Saganfactors} for references.
In view of Theorem \ref{f.th:charpoly}.1, when $\mathbbm{k} = \mathbb{R}$ these formulas lead to 
\[
a(A_{n-1}) = n!, \qquad 
a(BC_n) = 2^n n!, \qquad 
a(D_n) = 2^{n-1}n!,
\]
consistent with the general fact that the regions of a Coxeter arrangement are in bijection with the elements of the corresponding Coxeter group.

To compute the Tutte polynomials of the classical Coxeter arrangements, let the \emph{deformed exponential function} be 
$F(\alpha, \beta) = \sum_{n \geq 0} {\alpha^n \, \beta^{n \choose 2}}/{n!}$. 
The \emph{Tutte generating functions} of $A$ and $\Phi \in \{B,C,D\}$, defined as
\begin{eqnarray*}
{T}_A(X,Y,Z) &=& 1+X \sum_{n \geq 1} \overline{\chi}({A_{n-1}};X,Y) \frac{Z^n}{n!}, \\
{T}_\Phi(X,Y,Z) &=& \sum_{n \geq 0} \overline{\chi}({\Phi_n};X,Y) \frac{Z^n}{n!},
\end{eqnarray*}
are given by
\begin{eqnarray*}
T_A(X,Y,Z) &=& F(Z,Y)^X,\\
T_{BC}(X,Y,Z) &=& F(2Z,Y)^{(X-1)/2}F(YZ,Y^2),\\
T_D(X,Y,Z) &=& F(2Z,Y)^{(X-1)/2}F(Z,Y^2).
\end{eqnarray*}
Here we are following the convention that if $C$ and $A = 1+B$ are formal power series such that the constant coefficient of $A$ is $1$, then $A^C := e^{C \log(1+B)}$. This is well defined because the formal power series for $e^D$ and $\log(1+D)$ are well defined for any $D$ with constant coefficient equal to $0$.
The Tutte polynomials of the exceptional root systems are computed in \cite{f.DeConciniProcesi.Tutte, f.DeConciniProcesibook}.

\item \cite{f.MartinReiner} The formula above for the root systems of type $A$ gives the Tutte polynomials of the complete graphs; they are due to Tutte \cite{f.Tuttedichromatic}.
The coboundary polynomials of the \emph{complete bipartite graphs} $K_{m,n}$ are given by
\[
1 + X \sum_{{m,n \geq 0} \atop {(m,n) \neq (0,0)}} \overline{\chi}({K_{m,n}};X,Y) \frac{Z_1^m}{m!}\frac{Z_2^n}{n!} = \left(\sum_{m, n \geq 0} Y^{mn} \frac{Z_1^m}{m!} \frac{Z_2^n}{n!}\right)^X.
\]

\item
 \cite{f.BaranyReiner, f.Mphako}.
 Let $p$ be a prime power and consider 
the arrangement ${\mathcal{A}}(p,n)$ of \textbf{all} linear hyperplanes in ${\mathbb{F}}_p^n$. The characteristic polynomial is
\[
\chi({{\mathcal{A}}(p,n)};q) = (q-1)(q-p)(q-p^2) \cdots (q-p^{n-1}).
\]
The \emph{$p$-exponential generating function} of the coboundary polynomials is
\[
\sum_{n \geq 0} \overline{\chi}({{\mathcal{A}}(p,n)};X,Y) \frac{u^n}{(p;p)_n}
= 
\frac{(u;p)_\infty}{(Xu;p)_\infty}
\sum_{n \geq 0} Y^{1+p+\cdots + p^{n-1}} \frac{u^n}{(p;p)_n},
\]
where we define $(a;p)_n = (1-a)(1-pa) \cdots (1-p^{n-1}a)$ for $n \in \mathbb{N}$ and $(a;p)_\infty = (1-a)(1-pa)(1-p^2a)\cdots$.

\item
 \cite{f.ArdilaTutte} The \emph{threshold arrangement}  \index{hyperplane arrangement!threshold}  $\mathcal{T}_n$ in $\mathbbm{k}^n$ consists of the ${n \choose 2}$ hyperplanes $x_i + x_j = 0$ for $1 \leq i < j \leq n$. We have
\[
\sum_{n \geq 0} \overline{\chi}({\mathcal{T}_n};X,Y) \frac{Z^n}{n!} = 
\left(\sum_{r, s \geq 0} \frac{Y^{rs} Z^{r+s}}{r!s!}\right)^{(X-1)/2}  \left(\sum_{n \geq 0} \frac{Y^{n \choose 2} Z^n}{n!}\right).
\]
When $\mathbbm{k}=\mathbb{R}$, the regions of $\mathcal{T}_n$ are in bijection with the \emph{threshold graphs} \index{graph!threshold} on $[n]$. These are the graphs for which there exist vertex weights $w(i)$ for $1 \leq i \leq n$ and a \emph{threshold} $w$ such that edge $ij$ is present in the graph if and only if $w(i) + w(j) > w.$ They have many interesting properties and applications; see \cite{f.MahadevPeled}.

\item \cite{f.BrylawskiOxley}
 If ${\mathcal{A}}^{(k)}$ is the arrangement obtained from ${\mathcal{A}}$ by replacing each  hyperplane by $k$ copies of itself, then
\[
T({{\mathcal{A}}^{(k)}};x,y) = (y^{k-1} + \cdots + y^2 + y+1)^r T\left({\mathcal{A}};\frac{y^{k-1} + \cdots +y^2 + y+x}{y^{k-1} + \cdots + y^2 + y+1}, y^k\right).
\]
For arrangements with integer coefficients, this formula follows readily from the finite field method: notice that a point $p$ which is on $m$ hyperplanes of ${\mathcal{A}}$ is on $km$ hyperplanes of ${\mathcal{A}}^{(k)}$, and this implies that $\overline{\chi}({{\mathcal{A}}^{(k)}};X,Y) = \overline{\chi}({{\mathcal{A}}};X,Y^k)$. For a generalization, see Theorem \ref{f.th:multi}.

%
%
%
%
%
%
%
%

\item
 \cite{f.ArdilaTutte, f.PostnikovStanley}
There are many interesting \emph{deformations of the braid arrangement}, obtained by considering hyperplanes of the form $x_i-x_j=a$ for various constants $a$. 
Two particularly elegant ones are the \emph{Catalan} \index{hyperplane arrangement!Catalan} and \emph{Shi} \index{hyperplane arrangement!Shi} arrangements:
\begin{eqnarray*}
\mathrm{Cat}_{n-1} &:& x_i - x_j \in \{-1, 0, 1\}  \qquad (1 \leq i < j \leq n) \\
\mathrm{Shi}_{n-1} &:& x_i - x_j \in \{0, 1\}  \qquad (1 \leq i < j \leq n) 
\end{eqnarray*}
The left panel of Figure \ref{f.fig:arrangements} shows the arrangement ${\mathcal{A}}_2$ consisting of the  planes $x_1=x_2,\,  x_2=x_3,$ and $x_1=x_3$ in ${\mathbb{R}}^3$. Since all planes contain the line $x_1=x_2=x_3$, we quotient by it, obtaining a two-dimensional picture. The other panels show the Catalan and Shi arrangements.

\begin{figure}[ht]
 \begin{center}
  \includegraphics[scale=.8]{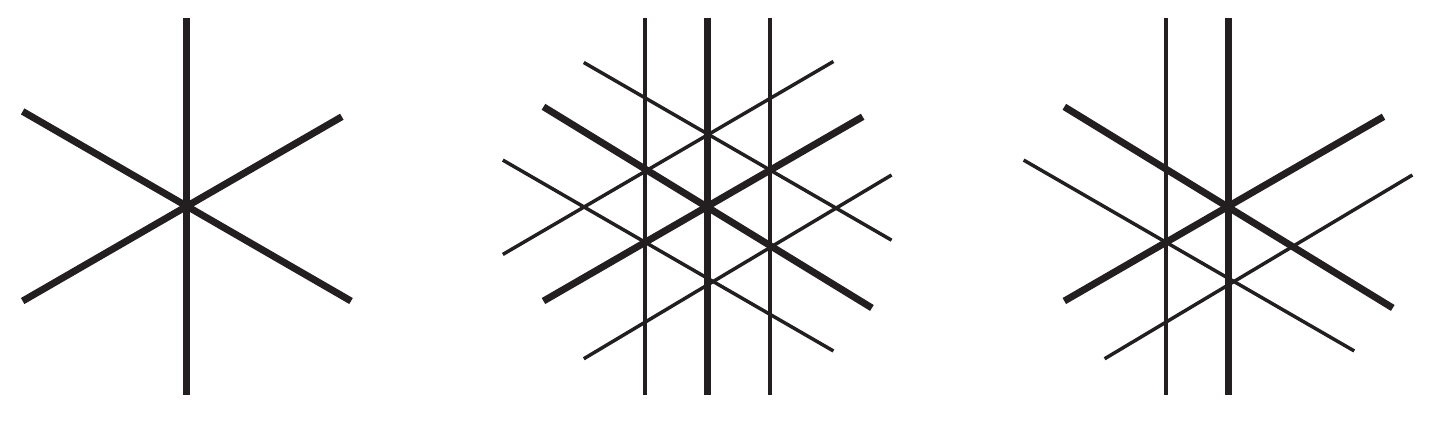}
  \caption{ \label{f.fig:arrangements}
  The braid, Catalan, and Shi arrangements ${\mathcal{A}}_{2}, \mathrm{Cat}_2$, and $\mathrm{Shi}_2$.}
  \end{center}
\end{figure}

When $\mathbbm{k} = \mathbb{R}$, we have the simple formulas
\begin{align*}
 a(\mathrm{Cat}_{n-1}) &=  n!C_n
& a(\mathrm{Shi}_{n-1}) =  (n+1)^{n-1}
\\
 b(\mathrm{Cat}_{n-1}) &=  n!C_{n-1}
& b(\mathrm{Shi}_{n-1})=  (n-1)^{n-1} 
\end{align*}
where $C_n=\frac1{n+1}{2n \choose n}$ is the $n$-th \emph{Catalan number}\index{Catalan numbers}, which famously has hundreds of different combinatorial interpretations \cite{f.StanleyCatalan}. The number $(n+1)^{n-1}$ also has many combinatorial interpretations of interest; \emph{parking functions} \index{parking function} are particuarly relevant  \cite{f.EC2}. We have
\begin{eqnarray*}
\chi({\mathrm{Cat}_{n-1}};q) &=&  q(q-n-1)(q-n-2)\cdots (q-2n+1), \\
\chi({\mathrm{Shi}_{n-1}};q) &=&  q(q-n)^{n-1}.
\end{eqnarray*}
There are substantially more complicated formulas for the Tutte polynomials of the Catalan and Shi arrangements \cite{f.ArdilaTutte}; it is not known whether they can be used to compute these polynomials efficiently.

\end{enumerate}

\section{{{Multivariate and Arithmetic Tutte Polynomials}}} \label{f.sec:variants}

We now discuss two useful variants of the Tutte polynomial.

\subsection{The multivariate Tutte polynomial}

The first variant is a refinement of the ordinary Tutte polynomial which is inspired by statistical mechanics. 

\begin{definition} \label{f.d:multivariateTutte} \cite{f.ArdilaTutte, f.Sokal}
The \emph{multivariate Tutte polynomial} \index{hyperplane arrangement!multivariate Tutte polynomial} of a hyperplane arrangement ${\mathcal{A}}$  is 
\[
\widetilde{Z}({\mathcal{A}};q, \mathbf{w}) = \sum_{
{{\mathcal{B}} \subseteq {\mathcal{A}}}\atop{{\mathcal{B}} \textrm{ central}}
}
q^{-r({\mathcal{B}})} \prod_{e \in {\mathcal{B}}} w_e 
\]
where $q$ and $(w_e)_{e \in {\mathcal{B}}}$ are indeterminates. 
\end{definition}

When ${\mathcal{A}}={\mathcal{A}}_G$ is a graphical arrangement, $\widetilde{Z}_{\mathcal{A}}(q; \mathbf{w})$ is equal to the partition function of the $q$-state Potts model on $G$; see also
\cite{f.Sokal}. Note that if we set $w_e=w$ for all $e$ in $\mathcal{A}$, then we have $\widetilde{Z}_{\mathcal{A}}(q,\mathbf{w}) = (w/q)^rT({\mathcal{A}};\frac{q}{w}+1, w+1)$, which is simply a transformation of the Tutte polynomial. 

\begin{theorem}\label{f.th:multi} \cite{f.ArdilaPostnikov}  For a vector ${\mathbf{a}} \in {\mathbb{N}}^n$, let ${\mathcal{A}}({\mathbf{a}})$ be the arrangement ${\mathcal{A}}$ where each hyperplane $e$ is replaced by $a_e$ copies of $e$. 
\begin{enumerate}
\item
The Tutte polynomial of  ${\mathcal{A}}({\mathbf{a}})$ is
\[
T({{\mathcal{A}}({\mathbf{a}})};x,y) = (x-1)^{r({\mathrm{supp }}({\mathbf{a}}))} \widetilde{Z}\left({\mathcal{A}}; (x-1)(y-1), y^{a_1-1}, \ldots y^{a_n-1}\right).
\]
\item
The generating function for 
the Tutte polynomials of \textbf{all} the arrangements ${\mathcal{A}}({\mathbf{a}})$ is essentially equivalent to the multivariate Tutte polynomial:
\begin{align*}
&\sum_{{\mathbf{a}} \in {\mathbb{N}}^n} \frac{T({{\mathcal{A}}({\mathbf{a}})};x,y)}{(x-1)^{r({\mathrm{supp }}({\mathbf{a}}))}} w_1^{a_1}\cdots w_n^{a_n}
= \\
& = \frac{1}{\prod_{i=1}^n(1-w_i)} \widetilde{Z}\left({\mathcal{A}}; (x-1)(y-1); \frac{(y-1)w_1}{1-yw_1}, \ldots,  \frac{(y-1)w_n}{1-yw_n}\right).
\end{align*}
Here ${\mathrm{supp }}({\mathbf{a}})$ denotes the set of hyperplanes $e$ for which $a_e > 0$. 

%
\end{enumerate}
\end{theorem}
There is also an algebraic manifestation of the multivariate Tutte polynomial: the multigraded Hilbert series of the \emph{zonotopal Cox ring} \index{hyperplane arrangement!zonotopal Cox ring} of ${\mathcal{A}}$ is an evaluation of the multivariate Tutte polynomial of ${\mathcal{A}}$  \cite{f.ArdilaPostnikov, f.SturmfelsXu}.

\subsection{The arithmetic Tutte polynomial} 

The second variant takes arithmetic into account, and is defined for vector arrangements, or for arrangements of subtori of codimension $1$ inside a torus.

\begin{definition}
\label{f.d:arithmeticTutte}
For a collection $A \subseteq {\mathbb{Z}}^d$ of integer vectors, the \emph{arithmetic Tutte polynomial} \index{arithmetic Tutte polynomial}
 is
\[
M({A};x,y)= \sum_{B \subseteq {A}} m(B)(x-1)^{r-r(B)}(y-1)^{|B|-r(B)}
\]
where, for each $B \subseteq A$, the \emph{multiplicity} $m(B)$ is the index of ${\mathbb{Z}} B$ as a sublattice of $({\mathrm{span }} \, B) \cap {\mathbb{Z}}^d$. If we use the vectors in $B$ as the  columns of a matrix, then $m(B)$ equals the greatest common divisor of the minors of full rank. The \emph{arithmetic characteristic polynomial} of $A$ is $(-1)^rq^{d-r}M(1-q,0)$. \index{arithmetic characteristic polynomial}
\end{definition}

There is also a multivariate arithmetic Tutte polynomial\index{arithmetic Tutte polynomial!multivariate}; see \cite{f.BrandenMoci}.

The next theorem shows that $M(A;x,y)$ encodes information about the zonotope of $A$; see also 
\cite{f.D'AdderioMoci.Ehrhart, f.DeConciniProcesibook, f.Stanleyzonotope}.

\begin{theorem}
Let $A \subseteq {\mathbb{Z}}^d$ be a set of integer vectors and let the \emph{zonotope} of $A$ be the Minkowski sum of the vectors in $A$; that is,
\[
Z(A) = \left\{\sum_{a \in A} \lambda_a a \, : \, 0 \leq \lambda_a \leq 1 \textrm{ for } a \in A\right\}.
\]
\begin{enumerate}
\item
The volume of the zonotope $Z(A)$ is $M(A;1,1)$.
\item
The zonotope $Z(A)$ contains $M(A;2,1)$ lattice points, $M(A;0,1)$ of which are in its interior. 
\item
The Ehrhart polynomial of the zonotope $Z(A)$, which counts the lattice points in the dilation $qZ(A)$ for $q \in \mathbb{N}$, equals $q^rM(A;1+\frac1q,1)$. 
\end{enumerate}
\end{theorem}

The artihmetic Tutte polynomial is also intimately related to the geometry of toric arrangements, as follows.

Let the torus $T=\mathrm{Hom}({\mathbb{Z}}^d,G)$ be the group of homomorphisms from ${\mathbb{Z}}^d$ to a multiplicative group $G$, such as the unit circle $\mathbb{S}^1$ or ${{\mathbbm{k}}}^*={{\mathbbm{k}}}\backslash \{0\}$ for a field ${{\mathbbm{k}}}$. 
The collection $A$ determines a \emph{toric arrangement} \index{toric arrangement} in $T$, consisting of the codimension 1 subtori
\[
T_a = \{t \in T \, : \, t(a) = 1\} \subset T
\]
for each vector $a \in A$.
For instance $a=(2,-3,5)$ gives the torus $x^2y^{-3}z^5=1$. 

The following results are the toric analogs of Theorems \ref{f.th:charpoly} and \ref{f.th:Tuttefinitefield} about hyperplane arrangements.

\begin{theorem}\label{f.t:hypertoric}
Let $A \subset {\mathbb{Z}}^d$  and $T=\mathrm{Hom}({\mathbb{Z}}^d,G)$ for a  group $G$. Consider the \emph{toric arrangement} of $A$ and its \emph{complement}\index{toric arrangement!complement}, namely,
\[
{\mathcal{T}}(A) = \{T_a \, : \, a \in A\}, \qquad R(A) = T \, \setminus \bigcup_{a \in {\mathcal{T}}(A)} T_a.
\]

\begin{enumerate}
\item \cite{f.EhrenborgReaddySlone, f.Moci.toric} If $G = \mathbb{S}^1$, the number of regions of $R({{A}})$ in the torus $(\mathbb{S}^1)^d$ is equal to  $M(A;1,0)$.

\item \cite{f.DeConciniProcesi.toric, f.DeConciniProcesibook, f.Moci.toric} If $G={\mathbb{C}}^*$, the Poincar\'e polynomial of $R({{A}})$ is equal to  $q^rM(A;2+\frac1q,0)$.

\item \cite{ f.ArdilaCastilloHenley, f.BrandenMoci} (Finite Field Method) If $G={\mathbb{F}}_{q+1}^*$ where $q+1$ is a prime power, then the number of elements of $R({{A}})$ is $(-1)^rq^{d-r}M(A;1-q,0)$, the \emph{arithmetic characteristic polynomial}. Furthermore,
\[
\sum_{p \in ({\mathbb{F}}^*_{q+1})^d} t^{h(p)} = (t-1)^r q^{d-r} M\left(A; \frac{q+t-1}{t-1}, t \right),
\]
where $h(p)$ is the number of hypertori of ${\mathcal{T}}(A)$ that $p$ lies on.
\end{enumerate}
\end{theorem}

\bigskip

As with ordinary Tutte polynomials, Theorem \ref{f.t:hypertoric}.3 may be used as a finite field method to compute arithmetic Tutte polynomials for some vector configurations and  toric arrangements. At the moment there are very few results along these lines. 

An important family that is well understood is the family of classical root systems, whose geometric properties motivate much of the theory of toric arrangements. Formulas for the arithmetic Tutte and characteristic polynomials of the classical root systems $A_n, B_n, C_n,$ and $D_n$ are given in \cite{f.ArdilaCastilloHenley}. Most of them resemble the formulas for the ordinary Tutte polynomials of the hyperplane arrangements ${\mathcal{A}}_n, \mathcal{BC}_n$, and $\mathcal{D}_n$ mentioned earlier. However, as should be expected, more subtle arithmetic issues arise -- especially in type $A$.

%
%
%
%

\small

\bibliographystyle{amsplain}
\bibliography{references}

\end{document}